\documentclass[smallextended]{svjour3}
\usepackage{amsfonts}
\usepackage{latexsym}
\usepackage{amsmath}
\usepackage{amssymb}
\usepackage{amscd}
\usepackage[affil-it]{authblk}
\usepackage{tikz}
\usetikzlibrary{backgrounds}
\usetikzlibrary{intersections}
\usepackage[affil-it]{authblk}
\newcommand*\edge[1][]{\draw[edge,#1]}
\tikzstyle{largegraph}=[
  every node/.style={circle,fill=black,inner sep=0pt,minimum width=6pt},
  edge/.style={thick},
  crossing/.style={rectangle,fill=red,inner sep=0pt,minimum size=3pt},
]
\newcounter{ngraphpaths}
\newcounter{ncrossings}
\tikzstyle{showcrossings}=[
  edge/.append style={increment_ngraphpaths,name path global={cross_p\the\value{ngraphpaths}}},
  increment_ngraphpaths/.code={\addtocounter{ngraphpaths}{1}},
  execute at begin picture={\setcounter{ngraphpaths}{0}\setcounter{ncrossings}{0}},
  execute at end picture={%
\ifnum \value{ngraphpaths}>1
\begin{scope}[on background layer]
  \pgfmathtruncatemacro\aend{\value{ngraphpaths}-1}
  \foreach \a in {1,...,\aend}{
    \pgfmathtruncatemacro\bstart{\a+1}
    \foreach \b in {\bstart,...,\value{ngraphpaths}}
      \path [name intersections={of={cross_p\a} and {cross_p\b},name=i,total=\t}]
        \ifnum \t>0
          \foreach \s in {1,...,\t}{(i-\s) node[crossing] {}}
          \pgfextra{\addtocounter{ncrossings}{\t}}
        \fi;
  }
\end{scope}\fi
\typeout{showcrossings: found \the\value{ncrossings} crossing(s)}}
]
\pgfkeys{/tikz/showcrossings/.value forbidden}
\addtocounter{MaxMatrixCols}{4}
\renewenvironment{proof}{\par {\sc {\bf Proof.}\hskip 5pt}}{\hfill \qed \par}
\newcommand{\undertilde}[1]{\ensuremath{\mathord{\vtop{\ialign{##\crcr
   $\hfil\displaystyle{#1}\hfil$\crcr\noalign{\kern1.5pt\nointerlineskip}
   $\hfil\tilde{}\hfil$\crcr\noalign{\kern1.5pt}}}}}}
\begin{document}

\title{There are no Cubic Graphs on 26 Vertices with Crossing Number 10 or 11}
\author{K. Clancy, M. Haythorpe, A. Newcombe and E. Pegg Jr}
\institute{Kieran Clancy
\at Flinders University\\
1284 South Road, Tonsley, Australia 5042\\
\email{kieran.clancy@flinders.edu.au}
\and
Michael Haythorpe - \emph{Corresponding Author}
\at Flinders University\\
1284 South Road, Tonsley, Australia 5042\\
\email{michael.haythorpe@flinders.edu.au}
\and
Alex Newcombe
\at Flinders University\\
1284 South Road, Tonsley, Australia 5042\\
\email{alex.newcombe@flinders.edu.au}
\and
Ed Pegg Jr
\at Wolfram Research\\
100 Trade Centre Drive, Champaign, IL 61820\\
\email{edp@wolfram.com}
}

\maketitle {\abstract We show that no cubic graphs of order 26 have crossing number larger than 9, which proves a conjecture of Ed Pegg Jr and Geoffrey Exoo that the smallest cubic graphs with crossing number 11 have 28 vertices. This result is achieved by first eliminating all girth 3 graphs from consideration, and then using the recently developed QuickCross heuristic to find good embeddings of each remaining graph. In the cases where the embedding found has 10 or more crossings, the heuristic is re-run with a different settings of parameters until an embedding with fewer than 10 crossings is found. We provide a minimal example of a cubic graph on 28 vertices with crossing number 10, and also exhibit for the first time a cubic graph on 30 vertices with crossing number 12, which we conjecture is minimal.}


\section{Introduction}\label{sec-Introduction}

In this manuscript we restrict our consideration to simple, undirected, connected graphs. Consider such a graph $G = \left<V, E\right>$ where $V$ is the set of vertices, and $E : V \rightarrow V$ is the set of edges. A {\em graph drawing} is a mapping of the vertices and edges to the plane. A graph is said to be {\em planar} if it is possible to produce a graph drawing for that graph such that edges only intersect at vertices. Henceforth, we will take {\em crossings} to describe only those edge intersections that occur away from vertices, and only permit two edges to cross in one place (so if three edges cross simultaneously, it is viewed as three individual crossings). Unquestionably, the seminal result in planarity is Kuratowski's theorem, which states that a graph is planar if and only if it contains no subdivisions of $K_5$ or $K_{3,3}$ \cite{kuratowski}.

If a graph is non-planar, then one may ask how close it is to being planar. One such measure is the {\em crossing number}. For any given drawing of a non-planar graph, there will be a positive number of crossings. Then the crossing number is the minimum number of crossings over all valid drawings of that graph. The crossing number of graph $G$ is denoted $cr(G)$. For example, the Petersen graph \cite{petersen} is shown in Figure \ref{fig-petersen} in two orientations. The first, standard, drawing has five crossings. However, the second drawing has only two, which is known to be the minimum possible for that graph.

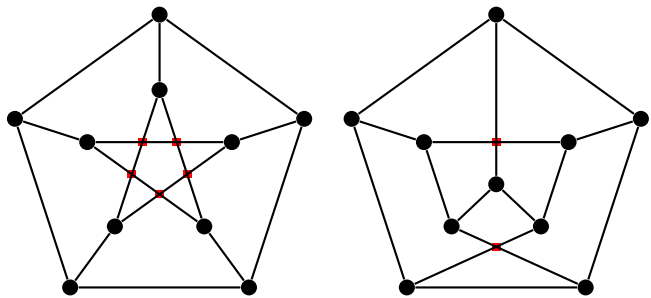
\begin{figure}[h!]
\begin{centering}
\begin{tikzpicture}[largegraph,scale=1,showcrossings]
\foreach \i [evaluate=\i as \a using 90-(\i-1)*360/5] in {1,...,5} {\node (o\i) at (\a:2) {}; \node (i\i) at (\a:1) {};}
\edge (o1) -- (o2) -- (o3) -- (o4) -- (o5) -- (o1);
\foreach \i in {1,...,5} \edge (o\i) -- (i\i);
\edge (i1) -- (i3) -- (i5); \edge (i1) -- (i4) -- (i2); \edge (i2) -- (i5);
\end{tikzpicture}\;\;\;\;\;%
\begin{tikzpicture}[largegraph,scale=1,showcrossings]
\foreach \i [evaluate=\i as \a using 90-(\i-1)*360/5] in {1,...,5} {\node (o\i) at (\a:2) {};}
\foreach \i [evaluate=\i as \a using 90-(\i-1)*360/5] in {2,...,5} {\node (i\i) at (\a:1) {};}
\node (c) at (270:0.25) {};
\edge (o1) -- (o2) -- (o3) -- (o4) -- (o5) -- (o1);
\edge (o2) -- (i2) -- (i5) -- (o5);
\edge (i2) -- (i3) -- (c) -- (i4) -- (i5);
\edge (o1) -- (c); \edge (o3) -- (i4); \edge (o4) -- (i3);
\end{tikzpicture}
\caption{The Petersen graph drawn with 5 crossings and a crossing minimal drawing with 2 crossings. \label{fig-petersen}}
\end{centering}
\end{figure}

The problem of finding the crossing number of a given graph is called the {\em crossing number problem}. It is NP-hard in general \cite{nphard}, but the problem is known to be fixed-parameter tractable \cite{kenichi}. For more information on this topic, we refer the interested reader to the excellent recent book by Schaefer \cite{schaeferbook} which details many aspects of the crossing number problem. There are also a number of useful surveys on this topic \cite{surveypaper,huangwang,schaefersurvey,vrto,winterbach}.

A question that may be asked is how many vertices are required for a graph to have a requested crossing number. If we restrict our consideration to cubic graphs (for which the crossing number problem is still NP-hard \cite{hlineny}), some results are known. In Table \ref{tab-cub} we list, for $k = 1, \hdots, 8$, the smallest order for which a cubic graph with crossing number $k$ exists, an example of one such cubic graph of that size, and the number of minimal examples; these are taken from Pegg Jr and Exoo \cite{exoo}. It should be noted that in \cite{exoo}, it was claimed that there are five minimal examples of cubic graphs with crossing number 8, however it has been subsequently determined during private communication between Pegg Jr and Eric Weisstein that two of them (labelled in \cite{exoo} as CNG 8D and CNG 8E) were erroneously listed, and the correct number is three.

\begin{table}[h]
\begin{center}
\hspace*{-0.7cm}\begin{tabular}{|c|l|l|c|}\hline
$k$ & Min $n$ & Example & \# of minimal examples\\
\hline 0 & 4 & $K_4$ & 1\\
\hline 1 & 6 & $K_{3,3}$ & 1\\
\hline 2 & 10 & Petersen graph & 2\\
\hline 3 & 14 & Heawood graph & 8\\
\hline 4 & 16 & M\"{o}bius-Kantor graph & 2\\
\hline 5 & 18 & Pappus graph & 2\\
\hline 6 & 20 & Desargues graph & 3\\
\hline 7 & 22 & Four (unnamed) graphs & 4\\
\hline 8 & 24 & McGee graph & 3\\
 \hline\end{tabular}
\end{center}
\caption{The minimum order of cubic graph for crossing number $k$, a famous example of each, and the number of minimal examples. For 7 crossings, none of the four minimal examples are famous graphs. All of the minimal examples are displayed in Pegg Jr and Exoo \cite{exoo}.}
\label{tab-cub}\end{table}

It is also worth noting at this point that the crossing numbers provided for many of the graphs listed in Table \ref{tab-cub}, although widely accepted as accurate and listed as such in numerous sources, have not been formally established in literature. We remedy that here, by using the excellent exact crossing minimisation solver of Chimani and Wiedera \cite{chimani} to confirm that all of the minimal examples (not just the named ones) listed in Pegg Jr and Exoo \cite{exoo} have their crossing numbers correctly listed, other than CNG 8D and CNG 8E as previously noted.

Results on minimal cubic graphs with crossing number larger than 8 have, to date, only been conjectured. It has been widely accepted that the Coxeter graph \cite{coxeter} on 28 vertices has crossing number 11, and the Levi graph \cite{levi} (also known as the Tutte-Coxeter graph) on 30 vertices has crossing number 13, although again these results have not been formally established in literature. We again remedy this oversight here by reporting that the exact solver \cite{chimani} confirms that these crossing numbers are accurate. Then, an open question posed by Pegg Jr and Exoo \cite{exoo} is whether any cubic graphs of order 26 have crossing number 11. More precisely, they conjectured the following.

\begin{conjecture}[Pegg Jr and Exoo \cite{exoo}]Define $a(k)$ to be the order of the smallest cubic graph with crossing number $k$. Then,

\begin{enumerate}\item[(i)] a(9) = a(10) = 26.
\item[(ii)] a(11) = 28.
\item[(iii)] a(13) = 30.\end{enumerate}
\label{conj}\end{conjecture}

Note that Pegg Jr and Exoo \cite{exoo} made no conjecture about the minimal size of cubic graphs with crossing number 12. In what follows, we answer Conjecture \ref{conj}(ii) in the affirmative, and so prove that the Coxeter graph is a minimal example of a cubic graph with crossing number 11. However, we also answer Conjecture \ref{conj}(i) in the negative, by showing that there are no cubic graphs on 26 vertices with crossing number 10. Finally, we exhibit a cubic graph on 30 vertices with crossing number 12, and conjecture that it is a minimal example of such a graph.

\section{Proving that Conjecture \ref{conj}(ii) is true}\label{sec-conjii}

We first seek to show that there are no cubic graphs on 26 vertices with crossing number 11. There are 2,094,480,864 cubic graphs on 26 vertices, up to isomorphism \cite{genreg}, that need to be considered. However, with the following simple argument, we can eliminate roughly 80\% of these.

\begin{proposition}No cubic graphs on 26 vertices with girth 3 have crossing number larger than 8.\label{prop-girth3}\end{proposition}

\begin{proof}Consider any cubic graph $G$ with girth 3. Then $G$ contains at least one triangle. Consider the three vertices in one such triangle. There are only three possibilities. Either none of the vertices are involved in a second triangle, two of vertices are, or all three are. These three cases are illustrated in Figure \ref{fig-girth3}. In the third case the graph must be $K_4$, which is planar. If none of the vertices are involved in a second triangle, it is clear that the triangle could be contracted to a single vertex without altering the crossing number, and what results is a smaller cubic graph with the same crossing number. In \cite{exoo} it was shown that no cubic graphs with fewer than 26 vertices have crossing number larger than 8, and hence the result follows immediately.

\begin{figure}[h!]
\begin{centering}
\begin{tikzpicture}[largegraph,scale=1.4]
\node (1) at (5,0) {}; \path (1) -- ++(0:1) node (2) {}; \path (1) -- ++(-120:1) node (3) {}; \path (1) -- ++(-60:1) node (4) {};
\edge (1) -- (2) -- (4) -- (3) -- (1) -- (4); \edge[bend right=60,looseness=1.5] (2) to (3);

\node (5) at (2.5,0) {}; \path (5) -- ++(0:1) node (6) {}; \path (5) -- ++(-120:1) node (7) {}; \path (5) -- ++(-60:1) node (8) {};
\edge (5) -- (6) -- (8) -- (7) -- (5) -- (8); \edge[densely dashed] (6) -- +(30:0.5); \edge[densely dashed] (7) -- +(210:0.5);

\node (9) at (0,0) {}; \path (9) -- ++(-120:1) node (10) {}; \path (9) -- ++(-60:1) node (11) {};
\edge (9) -- (10) -- (11) -- (9); \edge[densely dashed] (9) -- +(90:0.5); \edge[densely dashed] (10) -- +(210:0.5);
\edge[densely dashed] (11) -- +(-30:0.5);
\end{tikzpicture}
\caption{Possible triangle structures in a cubic graph.\label{fig-girth3}}
\end{centering}
\end{figure}
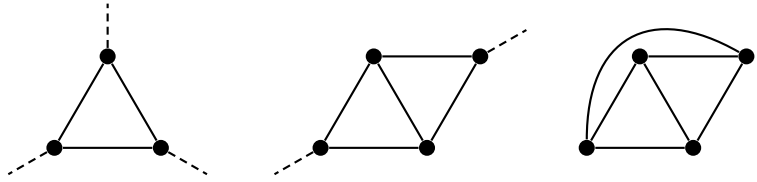

Then, the only remaining case is when two of the vertices are involved in a second triangle; we will show that in this case, one may also obtain a smaller cubic graph with the same crossing number, and hence the result from \cite{exoo} applies in this case as well. Because the graph is cubic, it is clear that the two triangles must form a diamond which can also be contracted to a single vertex without altering the crossing number. What results is a sub-cubic graph with a single vertex of degree 2. Denote this vertex $v$, and its two adjacent vertices $a$ and $b$. Then if edge $(a,b)$ is not present in the graph, vertex $v$ can be ``smoothed" in the sense that it is removed, and its two incident edges are merged into a single edge $(a,b)$, and a smaller cubic graph with the same crossing number results. If edge $(a,b)$ is present, it is clear that the sub-cubic graph also contains a triangle $a-b-v-a$ which can be contracted to a degree 2 vertex. This process can be repeated until a cubic graph results.\end{proof}

Proposition \ref{prop-girth3} indicates that we need only to consider cubic graphs on 26 vertices with girth at least 4. There are 432,757,568 such graphs up to isomorphism \cite{genreg}. For each of these graphs, we used the recently developed QuickCross heuristic \cite{quickcross}, a fast algorithm that finds valid embeddings which usually have an optimal or near-optimal number of crossings. For cubic graphs on 26 vertices, QuickCross can process several graphs per second, with the majority of that time used on read/write operations. We partitioned the graphs into sets of 50,000, resulting in roughly 8,000 individual jobs, each of which took up to 3 days to run, and distributed the jobs to over 400 cores on a High Performance Computer.

After the initial run with the default settings, embeddings with fewer than 11 crossings were obtained for 422,549,254 of the graphs, only leaving approximately ten million graphs to be considered more closely. By altering the parameters and random seed and re-running the remaining graphs, this number was further reduced to under 500,000 graphs. We continued this process, altering the parameters and random seeds until we were successful at obtaining an embedding with fewer than 11 crossings for every graph.

Next, we performed an additional test to ensure there was no errors arising from potential bugs with the QuickCross heuristic itself. For each graph, we took the embedding discovered by QuickCross and used it to construct the corresponding planarised graph. We then used the planarity algorithm of Hopcroft and Tarjan \cite{hopcroft} to confirm that the resulting graph was indeed planar. In this way, we verified that the embeddings found for each of the graphs were valid.

Note that we could also have eliminated any triangle-free 1-connected graphs from consideration; in the case of cubic graphs, these are all bridge graphs. It can be easily seen that for these graphs, the bridge (or any bridge, if there is more than one) can be removed, and the crossing number of the two remaining components can be computed and summed to give the crossing number of the original graph. Each of the two remaining components has a degree 2 vertex which can be safely removed, in the sense that its two incident edges are merged into a single edge. Since the original graph had no triangles, what remains are two cubic graphs containing $a$ and $24-a$ vertices respectively, for some even $4 \leq a \leq 20$. From Table \ref{tab-cub} it can be easily checked that triangle-free 1-connected cubic graphs of order 26 have crossing number no greater than 6. However, as we also wanted to use this experiment to test the robustness of QuickCross, we chose to include triangle-free 1-connected graphs in the experiment anyway.

The results of the experiment allow us to make the following claim.

\begin{theorem}There are no cubic graphs on 26 vertices with crossing number 11.\label{thm-cr11}\end{theorem}

It is trivial to see that Theorem \ref{thm-cr11} also implies there are no cubic graphs on fewer than 26 vertices with crossing number 11. Since the Coxeter graph on 28 vertices is known to have crossing number 11, this immediately answers Conjecture \ref{conj}(ii) in the affirmative.

\section{Proving that Conjecture \ref{conj}(i) is false}\label{sec-conji}

In \cite{exoo} it is seemingly implied, although not actually stated, that an example given of the McGee graph plus an edge (that is, the McGee graph where two particular edges are subdivided, and their new vertices joined by an edge) has crossing number 10. Indeed, this has been widely accepted as true, including being listed on the Online Encyclopedia of Integer Sequences \cite{oeis}. However, in Figure \ref{fig-26cross9} we provide two graph drawings of this graph. The first is a replica of the graph drawing given in \cite{exoo} which contains ten crossings, and the second is an alternative graph drawing with nine crossings. Hence, it is clear that the crossing number for this graph is at most nine. We have independently confirmed that the crossing number is indeed equal to nine for this graph, through the use of the exact crossing minimisation solver of Chimani and Wiedera \cite{chimani} which is available at http://crossings.uos.de.

\begin{figure}[h!]
\centering
\begin{tikzpicture}[largegraph,scale=0.7,showcrossings]
\clip (0.7,0.7) rectangle (8.3,9.3);
\node (1) at (1,9) {}; \node (2) at (7,9) {}; \node (3) at (2,8) {}; \node (4) at (6,8) {};
\node (5) at (8,8) {}; \node (6) at (3,7) {}; \node (7) at (4,7) {}; \node (8) at (6,7) {};
\node (9) at (3,6) {}; \node (10) at (5,6) {}; \node (11) at (7,6) {}; \node (12) at (2,5) {};
\node (13) at (3,5) {}; \node (14) at (4,5) {}; \node (15) at (5,5) {}; \node (16) at (2,4) {};
\node (17) at (3,4) {}; \node (18) at (4,4) {}; \node (19) at (6,4) {}; \node (20) at (2,3) {};
\node (21) at (4,3) {}; \node (22) at (5,3) {}; \node (23) at (7,3) {}; \node (24) at (3,2) {};
\node (25) at (1,1) {}; \node (26) at (8,1) {};
\edge (1) -- (2) -- (5) -- (4) -- (8) -- (7) -- (14) -- (15) -- (10) -- (11) -- (23);
\edge (1) -- (3) -- (7); \edge (4) -- (6) -- (9) -- (13) -- (12) -- (3);
\edge (5) -- (26) -- (25) -- (24) -- (17) -- (16) -- (20);
\edge (13) -- (18) -- (21) -- (22) -- (15);
\edge (18) -- (19) -- (23) -- (24); \edge (2) -- (11);
\edge (14) -- (17); \edge (8) -- (19); \edge (22) -- (26);
\edge (9) -- (10); \edge (20) -- (21); \edge (6) -- (16);
\edge[bend right=10] (12) to (25); \edge[bend right=10] (1) to (20);
\end{tikzpicture}  \;\;\; \begin{tikzpicture}[largegraph,scale=0.7,showcrossings]
\clip (-1,0.7) rectangle (8.5,9.3);
\node (1) at (1,9) {}; \node (2) at (7,9) {}; \node (3) at (1,8) {}; \node (4) at (7,8) {};
\node (5) at (1,7) {}; \node (6) at (2,7) {}; \node (7) at (7,7) {}; \node (8) at (2,6) {};
\node (9) at (3,6) {}; \node (10) at (6,6) {}; \node (11) at (7,6) {}; \node (12) at (3,5) {};
\node (13) at (4,5) {}; \node (14) at (2,4) {}; \node (15) at (3,4) {}; \node (16) at (4,4) {};
\node (17) at (5,4) {}; \node (18) at (6,4) {}; \node (19) at (7,4) {}; \node (20) at (3,3) {};
\node (21) at (6,3) {}; \node (22) at (3,2) {}; \node (23) at (4,2) {}; \node (24) at (5,2) {};
\node (25) at (1,1) {}; \node (26) at (7,1) {};
\edge (1) -- (3) -- (5) -- (6) -- (8) -- (9) -- (12) -- (13) -- (16) -- (17) -- (18) -- (19) -- (11) -- (7) -- (4) -- (2) -- (1);
\edge (8) -- (14) -- (15) -- (20) -- (22) -- (23) -- (24) -- (26) -- (25) -- (14); \edge (3) -- (10) -- (11); \edge (10) -- (13);
\edge (4) -- (9); \edge (17) -- (21) -- (23); \edge (19) -- (26); \edge (15) -- (16);
\edge[bend left=10] (6) to (18); \edge[bend right,looseness=0.5] (1) to (25); \edge[out=270,looseness=1.5] (5) to (22);
\edge[out=240,in=210,looseness=2] (12) to (24); \edge[out=290,in=10,looseness=2] (7) to (20);
\edge[out=290,in=0,looseness=1.2] (2) to (21);
\end{tikzpicture}
\caption{Two drawings of the McGee graph plus an edge. The first is the drawing displayed in \cite{exoo} with ten crossings, and the second is a drawing with nine crossings. \label{fig-26cross9}}
\end{figure}
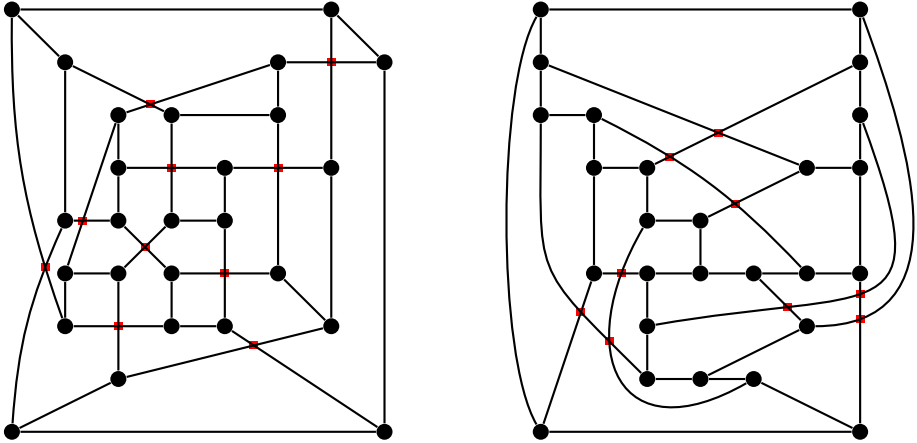

It should be noted that in \cite{exoo}, the authors exhaustively catalogued the crossing-maximal cubic graphs with up to 24 vertices, and found none with crossing number 9. Hence, the McGee graph plus an edge is a minimal example of such a graph, and we can conclude that $a(9) = 26$ as predicted in Conjecture \ref{conj}(i).

However, given that the McGee graph plus an edge is not an example of a cubic graph on 26 vertices with crossing number 10, we wondered if there were any other such examples. Obviously, the 80\% of graphs which were shown in Section \ref{sec-conjii} to have crossing number 8 or less do not provide an example. Furthermore, although the experiment conducted in Section \ref{sec-conjii} was willing to accept embeddings with precisely 10 crossings, in practice QuickCross returned embeddings with fewer than 10 crossings for approximately 97\% of the graphs, so these graphs also do not provide an example. At the conclusion of the experiment in Section \ref{sec-conjii}, we were left with 13,335,262 graphs for which the graph drawing obtained by QuickCross had precisely ten crossings. Of course, this is not to say that their crossing number is ten, just that the first example of a graph drawing with fewer than eleven crossings that was found by QuickCross had ten crossings.

We resubmitted each of these graphs to QuickCross, this time demanding that it continue to try different random seeds until a graph drawing with fewer than ten crossings was found. Having faith in Conjecture \ref{conj}(i), we assumed that we would find some graphs where this proved to be impossible. However, after a few days, QuickCross was successful at doing so for all 13,335,262 graphs. We further confirmed these results using the same check as described in Section \ref{sec-conjii}. The results of this experiment allow us to make the following, somewhat surprising, claim.

\begin{theorem}There are no cubic graphs on 26 vertices with crossing number 10.\label{thm-cr10}\end{theorem}

Again, it is trivial to see that Theorem \ref{thm-cr10} also implies that are no cubic graphs on fewer than 26 vertices with crossing number 10, and hence we can conclude that $a(10) > 26$, which implies that Conjecture \ref{conj}(i) is false. The set of ordered edge crossings, corresponding to a valid embedding, for each of the cubic graphs on 26 vertices with girth at least four is available upon request.

We now provide an instance of a cubic graph on 28 vertices which does have crossing number 10, drawn optimally in Figure \ref{fig-2810}.

\begin{figure}[h!]
\centering
\begin{tikzpicture}[largegraph,scale=0.6,showcrossings]
\node (1) at (4,13) {}; \node (2) at (2,12) {}; \node (3) at (3,12) {}; \node (4) at (4,12) {};
\node (5) at (3,11) {}; \node (6) at (7,11) {}; \node (7) at (3,10) {}; \node (8) at (6,9) {};
\node (9) at (7,9)  {}; \node (10) at (8,9) {}; \node (11) at (6,8) {}; \node (12) at (7,8) {};
\node (13) at (3,7) {}; \node (14) at (4,7) {}; \node (15) at (2,6) {}; \node (16) at (3,6) {};
\node (17) at (4,6) {}; \node (18) at (6,6) {}; \node (19) at (7,6) {}; \node (20) at (7,5) {};
\node (21) at (8,5) {}; \node (22) at (1,3) {}; \node (23) at (2,3) {}; \node (24) at (3,3) {};
\node (25) at (8,3) {}; \node (26) at (1,2) {}; \node (27) at (8,2) {}; \node (28) at (1,1) {};
\edge (1) -- (4) -- (3) -- (2); \edge (3) -- (5) -- (7); \edge (8) -- (11) -- (12) -- (9);
\edge (28) -- (26) -- (22) -- (23) -- (24); \edge (25) -- (27) -- (28); \edge (1) -- (6) -- (10);
\edge (23) -- (15) -- (16) -- (13) -- (14) -- (17) -- (18) -- (19) -- (20) -- (21);
\edge[bend left] (10) to (27);
\edge (24) -- (25); \edge[bend right=10] (26) to (21);
\edge (17) -- (24); \edge[bend right=15] (16) to (20);
\edge[bend right=70,looseness=1.5] (1) to (15); \edge (2) -- (13);
\edge[bend right=30] (2) to (28); \edge[bend right=10] (7) to (22);
\edge (5) -- (18); \edge  (7) -- (8); \edge  (11) -- (14);
\edge (4) -- (9); \edge  (6) -- (8); \edge  (10) -- (19);
\edge (12) -- (21); \edge[bend left] (9) to (25);
\end{tikzpicture}
\caption{An optimal drawing of a minimal example of a cubic graph with crossing number 10. \label{fig-2810}}
\end{figure}
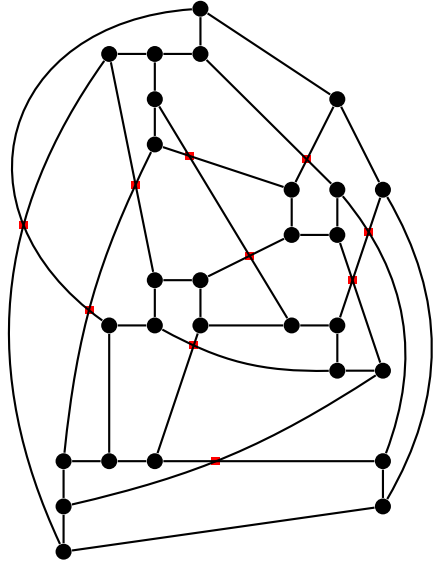

To the best of the authors' knowledge, the graph in Figure \ref{fig-2810} is not a famous graph. We have confirmed that the crossing number is 10 through the use of Chimani and Wiedera's solver \cite{chimani}. Hence, we can conclude that $a(10) = 28$.

\section{Cubic graphs with larger crossing numbers}

When considering the famous examples listed in Table \ref{tab-cub}, one can make the following observation: for each choice of $k = 1, \hdots, 8$, there exists a minimal cubic graph with crossing number $k$ that is also girth-maximal for cubic graphs of that order. This is also the case for the McGee graph plus an edge ($k = 9$), the graph in Figure \ref{fig-2810} ($k = 10$), the Coxeter graph ($k = 11$) and the Levi graph ($k = 13$). This observation leads us to propose the following conjecture.

\begin{conjecture}Define $c(n)$ to be the maximum crossing number among all cubic graphs of order $n$, and $g(n)$ to be the maximum girth among all cubic graphs of order $n$. Then, for each $n = 4, 6, 8, \hdots$ there exists a cubic graph of order $n$ which has girth $g(n)$ and crossing number $c(n)$.\label{conj-maximal}\end{conjecture}

For $n = 8$, the Wagner graph is an example, and for $n = 12$, the Twinplex graph \cite{robertson} is an example. These, along with the examples provided above, verify Conjecture \ref{conj-maximal} for all even $4 \leq n \leq 30$.

We now turn our attention to cubic graphs with crossing number 12. As shown previously, $a(11) = 28$ and $28 \leq a(13) \leq 30$. It seems that $a(12)$ is likely to be either 28 or 30, and given the compelling evidence for Conjecture \ref{conj-maximal}, it seems appropriate to search for an example of a 12-crossing cubic graph among the set of high-girth cubic graphs.

First, we considered 28-vertex cubic graphs. For cubic graphs of this order, the maximum girth is 7, and there are 21 such graphs (including the Coxeter graph). We used QuickCross to search for drawings with fewer than 12 crossings, and it was successful in doing so in all cases. Although this does not prove definitively that no cubic graphs on 28 vertices have crossing number 12, it does provide solid evidence given Conjecture \ref{conj-maximal}.

We then turned our attention to 30-vertex cubic graphs. There is a single graph with girth 8 (the Levi graph) which has crossing number 13, and then 545 graphs with girth 7. Again, we used QuickCross to search for drawings of these 545 graphs with fewer than 12 crossings, and it was successful in doing so for 544 of the graphs. For the one remaining graph, we used \cite{chimani} to confirm that it does have crossing number 12. Although it does not appear to be a famous graph, it can be viewed as the result of excising an edge from the Levi graph, and then inserting a new edge. We display this graph in Figure \ref{fig-3012}, in one of its optimal embeddings.

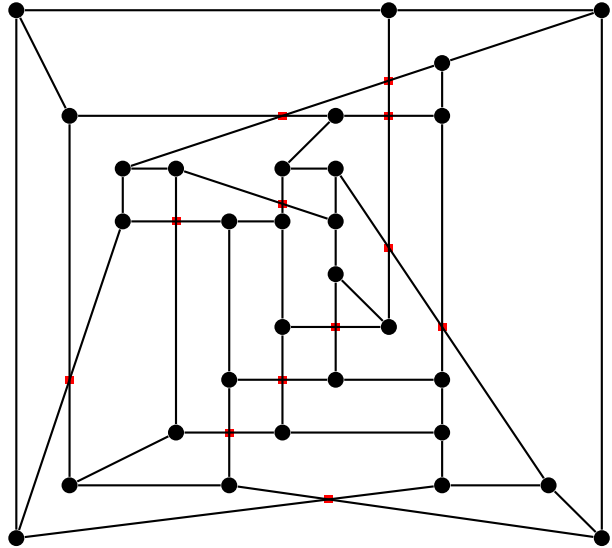
\begin{figure}[h!]\begin{center}\begin{tikzpicture}[largegraph,scale=0.7,showcrossings]
\node (1) at (10,1) {}; \node (2) at (8,9) {}; \node (3) at (11,10) {}; \node (4) at (1,8) {}; \node (5) at (6,8) {};
\node (6) at (8,8) {}; \node (7) at (2,7) {}; \node (8) at (3,7) {}; \node (9) at (5,7) {}; \node (10) at (6,7) {};
\node (11) at (2,6) {}; \node (12) at (4,6) {}; \node (13) at (5,6) {}; \node (14) at (6,6) {}; \node (15) at (4,3) {};
\node (16) at (6,5) {}; \node (17) at (5,4) {}; \node (18) at (6,3) {}; \node (19) at (8,3) {}; \node (20) at (3,2) {};
\node (21) at (5,2) {}; \node (22) at (8,2) {}; \node (23) at (7,4) {}; \node (24) at (1,1) {}; \node (25) at (4,1) {};
\node (26) at (11,0) {}; \node (27) at (0,0) {}; \node (28) at (8,1) {}; \node (29) at (7,10) {}; \node (30) at (0,10) {};
\edge (1) -- (10); \edge (1) -- (26); \edge (1) -- (28); \edge (2) -- (3); \edge (2) -- (6);
\edge (2) -- (7); \edge (3) -- (26); \edge (3) -- (29); \edge (4) -- (5); \edge (4) -- (24);
\edge (4) -- (30); \edge (5) -- (6); \edge (5) -- (9); \edge (6) -- (19); \edge (7) -- (8);
\edge (7) -- (11); \edge (8) -- (14); \edge (8) -- (20); \edge (9) -- (10); \edge (9) -- (13);
\edge (10) -- (14); \edge (11) -- (12); \edge (11) -- (27); \edge (12) -- (13); \edge (12) -- (15);
\edge (13) -- (17); \edge (14) -- (16); \edge (15) -- (18); \edge (15) -- (25); \edge (16) -- (18);
\edge (16) -- (23); \edge (17) -- (21); \edge (17) -- (23); \edge (18) -- (19); \edge (19) -- (22);
\edge (20) -- (21); \edge (20) -- (24); \edge (21) -- (22); \edge (22) -- (28); \edge (23) -- (29);
\edge (24) -- (25); \edge (25) -- (26); \edge (27) -- (28); \edge (27) -- (30); \edge (29) -- (30);
\end{tikzpicture}\caption{The cubic graph on 30 vertices with crossing number 12, drawn in one of its optimal embeddings.\label{fig-3012}}\end{center}\end{figure}

Given the above discussion, we propose the following conjecture. If there is a counterexample to Conjecture \ref{conj-final}, it must have 28 vertices and girth 4, 5 or 6.

\begin{conjecture}The graph shown in Figure \ref{fig-3012} is a minimal example of a cubic graph with crossing number 12.\label{conj-final}\end{conjecture}

Finally, we conclude this paper with Table \ref{tab-cub2}, which contains a list of the smallest known cubic graphs with crossing numbers 9 and above, reproduced from \cite{weisstein} where these findings were recently summarised. The graph labelled $GP(13,5)$ is one of the generalized Petersen graphs \cite{watkins}. The edge-excised Coxeter and Levi graphs are displayed in \cite{weisstein}.

\begin{table}[h!]
\begin{center}
\hspace*{-0.7cm}\begin{tabular}{|c|l|l|}\hline
$k$ & Min $n$ known & Examples \\
\hline  &  & GP(13,5)\\
9 & 26 & Edge-excised Coxeter Graph\\
& & Graph shown in Figure \ref{fig-26cross9} \\
\hline 10 & 28 & Edge-excised Levi graph\\
& & Graph shown in Figure \ref{fig-2810} \\
\hline 11 & 28 & Coxeter graph\\
\hline 12 & 30 & Graph shown in Figure \ref{fig-3012}\\
\hline 13 & 30 & Levi Graph \\
 \hline\end{tabular}
\end{center}
\caption{The known smallest instances of cubic graphs with crossing numbers between 9 and 13. The examples with 9, 10 and 11 crossings are known to be minimal, the other cases are only conjectured to be minimal.}
\label{tab-cub2}\end{table}

\begin{acknowledgements}We are greatly indebted to Tilo Wiedera and Markus Chimani for their tireless patience in manually running their exact crossing minimisation solver on many of the more difficult instances in this paper. We also thank Eric Weisstein for numerous useful conversations that significantly improved this paper.\end{acknowledgements}


\begin{thebibliography}{99}
\bibitem{chimani} Chimani, M. and Wiedera, T.  An ILP-based Proof System for the Crossing Number Problem. In: {\em 24th European Symposium of Algorithms (ESA) 2016}, Aarhus, Denmark, Leibniz. Int. Prov. Inform. 56:29.1--29.13, 2016.
\bibitem{quickcross} Clancy, K., Haythorpe, M. and Newcombe, A. \lq\lq An effective crossing minimisation heuristic based on star insertion", {\em Journal of Graph Algorithms and Applications}, 23(2):135--166, 2019.
\bibitem{surveypaper} Clancy, K., Haythorpe, M. and Newcombe, A.  \lq\lq A survey of graphs with known or bounded crossing numbers", {\em submitted to Australasian Journal of Combinatorics}. Available at: https://arxiv.org/abs/1901.05155.
\bibitem{coxeter} Coxeter, H.S.M. \lq\lq My graph", {\em Proceeding of the London Mathematical Society}, 3(1):117--136, 1983.
\bibitem{nphard} Garey, M.R., Johnson, D.S. \lq\lq Crossing number is NP-complete", {\em SIAM Journal on Algebraic Discrete Methods}, 4(3):312--316, 1983.
\bibitem{hlineny} Hlin\v{e}ny, P. \lq\lq Crossing number is hard for cubic graphs", {\em Journal of Combinatorial Theory, Series B}, 96(4):455--471, 2006.
\bibitem{petersen} Holton, D.A., Sheehan, J. \lq\lq The Petersen Graph", Cambridge University Press, 1993.
\bibitem{huangwang} Huang, Y. and Wang, J.  \lq\lq Survey of the crossing number of graphs", {\em Journal of East China Normal University of Natural Science}, 2010(3):68--80, 2010.
\bibitem{hopcroft} J. Hopcroft and R. Tarjan. \lq\lq Efficient planarity testing", {\em Journal of the ACM}, 21(4):549--568, 1974.
\bibitem{kenichi} Kawarabayashi, K., Reed, B. \lq\lq Computing crossing number in linear time" {\em Proceedings of the 29th Annual ACM Symposium on Theory of Computing}, pp.382--390, 2007.
\bibitem{kuratowski} Kuratowski, C. \lq\lq Sur le probleme des courbes gauches en topologie", {\em Fundamenta Mathematicae}, 15(1):271--283, 1930.
\bibitem{levi} Levi, F.W. \lq\lq Finite geometrical systems", University of Calcutta, 1942.
\bibitem{genreg} Meringer, M. \lq\lq Fast Generation of Regular Graphs and Constructions of Cages", \emph{Journal of Graph Theory}, 30:137--146, 1999.
\bibitem{exoo} Pegg Jr, E. and Exoo, G. \lq\lq Crossing Number Graphs", {\em The Mathematica Journal}, 11(2):161--170, 2009.
\bibitem{robertson} Robertson, N., Seymour, P. and Thomas, R. \lq\lq Excluded minors in cubic graphs", {\em Journal of Combinatorial Theory, Series B}, 2019, https://doi.org/10.1016/j.jctb.2019.02.002
\bibitem{schaefersurvey} Schaefer, M. \lq\lq The graph crossing number and its variants: A survey", {\em Electronic Journal of Combinatorics}, DS21, 2017.
\bibitem{schaeferbook} Schaefer, M.  {\em Crossing numbers of graphs,} CRC Press, 2018.
\bibitem{oeis} Sloane, N.J.A. \lq\lq A110507: Number of nodes in the smallest cubic graph with crossing number $n$", {\em The On-line Encyclopedia of Integer Sequences}, 2007. Available at: https://oeis.org/A110507
\bibitem{vrto} Vrt'o, I.  \lq\lq Bibliography on crossing numbers", Available at: ftp://ftp.ifi.savba.sk/pub/imrich/crobib.pdf, last updated 2014.
\bibitem{watkins} Watkins, M.E.  \lq\lq A Theorem on Tait Colorings with an Application to the Generalized Petersen Graphs",  {\em Journal of Combinatorial Theory}, 6:152--164, 1969.
\bibitem{weisstein} Weisstein, E.W. \lq\lq Smallest Cubic Crossing Number Graphs". From: {\em Mathworld--A Wolfram Web Resource}. http://mathworld.wolfram.com/SmallestCubicCrossingNumberGraph.html, updated April 10th 2019.
\bibitem{winterbach} Winterbach, W. {\em The crossing number of a graph in the plane}. Master's Thesis, University of Stellenbosch, 2005.
\end{thebibliography}
\end{document}